\documentclass[11pt]{article}

\usepackage{amsmath,amsthm,amsfonts}
\usepackage{amssymb}
\usepackage{hyperref}

\allowdisplaybreaks

\begin{document}

\title{Some determinants involving incomplete Fubini numbers
}  
\author{
Takao Komatsu\\
\small School of Mathematics and Statistics\\[-0.8ex]
\small Wuhan University\\[-0.8ex]
\small Wuhan 430072 China\\[-0.8ex]
\small \texttt{komatsu@whu.edu.cn}\\\\
Jos\'e L. Ram\'{\i}rez\\
\small Departamento de Matem\'aticas\\[-0.8ex]
\small Universidad Nacional de Colombia\\[-0.8ex]
\small Bogot\'a, Colombia\\[-0.8ex] 
\small \texttt{jlramirezr@unal.edu.co} 
}

\date{
\small MR Subject Classifications: Primary 11C20; Secondary 15B05, 05A18, 11B75.
}

\maketitle

\def\fl#1{\left\lfloor#1\right\rfloor}
\def\cl#1{\left\lceil#1\right\rceil}
\def\ang#1{\left\langle#1\right\rangle}
\def\stf#1#2{\left[#1\atop#2\right]} 
\def\sts#1#2{\left\{#1\atop#2\right\}}

\newtheorem{theorem}{Theorem}
\newtheorem{Prop}{Proposition}
\newtheorem{Cor}{Corollary}
\newtheorem{Lem}{Lemma}

\begin{abstract}
We study some properties of restricted and associated Fubini numbers.  In particular, they have the natural extensions of the original Fubini numbers in the sense of determinants.  We also introduce modified Bernoulli and Cauchy numbers and study characteristic properties.    
\end{abstract}


\section{Introduction}

The Fubini numbers (or the ordered Bell numbers) form an integer sequence in which the $n$th term counts the number of weak orderings of a set with $n$ elements.
The Stirling numbers of the second kind $\sts{n}{k}$ counts the number of ways that the set $A=\{1,2,\dots,n\}$ with $n$ elements can be partitioned into $k$ non-empty subsets.
Fubini numbers are defined by
$$
F_n=\sum_{k=0}^n k!\sts{n}{k}
$$
(\cite{Mezo}).
They can be expanded involving binomial coefficients or given by an infinite series.
$$
F_n=\sum_{k=0}^n\sum_{j=0}^k(-1)^{k-j}\binom{k}{j}j^n=\frac{1}{2}\sum_{m=0}^\infty\frac{m^n}{2^m}\,.
$$
The first Fubini numbers are presented as
$$
\{F_n\}_{n=0}^\infty=\{1, 1, 3, 13, 75, 541, 4683, 47293, 545835, 7087261, 102247563, \dots \}
$$
(\cite[A000670]{oeis}).
The (exponential) generating function of Fubini numbers is given by
\begin{equation}
\frac{1}{2-e^t}=\sum_{n=0}^\infty F_n\frac{t^n}{n!}\,.
\label{fubini:gf}
\end{equation}
The Fubini numbers satisfy the recurrence relation:
\begin{equation}
F_n=\sum_{j=1}^n\binom{n}{j}F_{n-j}\,.
\label{fubini:rec}
\end{equation}

The literature contains several generalizations of Stirling numbers; see \cite{mansour-2015a}. Among them, the so-called restricted and associated Stirling numbers of both kinds (cf. \cite{bona-2016a, choijy-2006b,
comtet-1974a, KLM,  KMS,  KomRam, Mezo, Moll}).  The restricted Stirling number of the second kind $\sts{n}{k}_{\le m}$ gives the number of partitions of $n$ elements into $k$ subsets under the restriction that none of the blocks contain more than $m$ elements.  The associated Stirling number of the second kind $\sts{n}{k}_{\ge m}$ gives the number of partitions of an $n$ element set into $k$ subsets such that every block contains at least $m$ elements.

In this paper, we are  interested in the sequence of the number of weak orderings of a set with $n$ elements in blocks with at most (and at least) $m$ elements. These two integer sequences are called the  restricted Fubini  numbers and associated Fubini numbers, respectively.  In particular, we give some  new determinant expressions for these sequences, and for some related sequences such as the modified Bernoulli and Cauchy numbers. We also apply the Trudi's formula to obtain new explicit combinatorial formulas for these sequences.

Recently, studies on the determinantal representations  involving interesting combinatorial sequences appeared.  For example, Bozkurt and Tam \cite{BT} established  formulas for the determinants and inverses of  $r$-circulant matrices of  second-order linear recurrences.  Li and MacHenry \cite{LM} gave a general method for finding permanental and determinantal representations of many families of integer sequences such as Fibonacci and Lucas polynomials. In \cite{Li}, Li obtained three new Fibonacci-Hessenberg matrices and studied its relations with Pell and Perrin sequence.  Xu and Zhou \cite{XZ}  studied a determinantal representation for a generalization of the Stirling numbers, called $r$-Whitney numbers, see also \cite{MR}. More examples can be found in
\cite{SR,  TX,  YB}.

\section{Determinants}

In \cite{Glaisher}, Glaisher gave determinant expressions of several numbers, including Bernoulli numbers $B_n$, Cauchy numbers $c_n$ and Euler numbers $E_n$, defined by
$$
\frac{t}{e^t-1}=\sum_{n=0}^\infty B_n\frac{t^n}{n!},\quad
\frac{t}{\ln(1+t)}=\sum_{n=0}^\infty c_n\frac{t^n}{n!}\quad\text{and}\quad
\frac{1}{\cosh t}=\sum_{n=0}^\infty E_n\frac{t^n}{n!}\,,
$$
respectively.
A determinant expression of the classical Bernoulli numbers is given by
\begin{equation}
B_n=(-1)^n n!\left|
\begin{array}{ccccc}
\frac{1}{2!}&1&&&\\
\frac{1}{3!}&\frac{1}{2!}&&&\\
\vdots&\vdots&\ddots&1&\\
\frac{1}{n!}&\frac{1}{(n-1)!}&\cdots&\frac{1}{2!}&1\\
\frac{1}{(n+1)!}&\frac{1}{n!}&\cdots&\frac{1}{3!}&\frac{1}{2!}
\end{array}
\right|\,.
\label{ber:det}
\end{equation}

A determinant expression of Cauchy numbers is given by
\begin{equation}
c_n=n!\left|
\begin{array}{ccccc}
\frac{1}{2}&1&&&\\
\frac{1}{3}&\frac{1}{2}&&&\\
\vdots&\vdots&\ddots&1&\\
\frac{1}{n}&\frac{1}{n-1}&\cdots&\frac{1}{2}&1\\
\frac{1}{n+1}&\frac{1}{n}&\cdots&\frac{1}{3}&\frac{1}{2}
\end{array}
\right|\,.
\label{cau:det}
\end{equation}
The value of this determinant, that is, $b_n=c_n/n!$ are called Bernoulli numbers of the second kind.

A determinant expression of Euler numbers is given by
\begin{equation}
E_{2n}=(-1)^n (2n)!
\begin{vmatrix}
\frac{1}{2!}& 1 &~& ~&~\\
\frac{1}{4!}&  \frac{1}{2!} & 1 &~&~\\
\vdots & ~  &  \ddots~~ &\ddots~~ & ~\\
\frac{1}{(2n-2)!}& \frac{1}{(2n-4)!}& ~&\frac{1}{2!} &  1\\
\frac{1}{(2n)!}&\frac{1}{(2n-2)!}& \cdots &  \frac{1}{4!} & \frac{1}{2!}
\end{vmatrix}\,.
\label{eul:det}
\end{equation}

Several hypergeometric numbers can be also expressed in similar ways (\cite{AK1,AK2,KY}).
However, all numbers cannot be expressed in such simpler determinant expressions. In general, we have the following relation:

\begin{theorem} \label{propo1}
Let $\{\alpha_n\}_{n\ge 0}$ be a sequence with $\alpha_0=1$, and $R(j)$ be a function independent of $n$. Then
\begin{equation}
\alpha_n=\left|
\begin{array}{ccccc}
R(1)&1&&&\\
R(2)&R(1)&&&\\
\vdots&\vdots&\ddots&1&\\
R(n-1)&R(n-2)&\cdots&R(1)&1\\
R(n)&R(n-1)&\cdots&R(2)&R(1)
\end{array}
\right|\,.
\label{10a}
\end{equation}
if and only if
\begin{equation}
\alpha_n=\sum_{j=1}^n(-1)^{j-1}R(j)\alpha_{n-j}\quad(n\ge 1)
\label{10b}
\end{equation}
with $\alpha_0=1$.
\label{det}
\end{theorem}
\begin{proof}
When $n=1$, the result is clear because $\alpha_1=R(1)$.
Assume that (\ref{10b}) is valid for any positive integer $n$.  Suppose that (\ref{10a}) is true up to $n-1$.  Then by expanding the first row, we have
\begin{multline*}
\left|
\begin{array}{ccccc}
R(1)&1&&&\\
R(2)&R(1)&&&\\
\vdots&\vdots&\ddots&1&\\
R(n-1)&R(n-2)&\cdots&R(1)&1\\
R(n)&R(n-1)&\cdots&R(2)&R(1)
\end{array}
\right|\\
=R(1)\alpha_{n-1}
-\left|
\begin{array}{ccccc}
R(2)&1&&&\\
R(3)&R(1)&&&\\
\vdots&\vdots&\ddots&1&\\
R(n-1)&R(n-3)&\cdots&R(1)&1\\
R(n)&R(n-2)&\cdots&R(2)&R(1)
\end{array}
\right|\\
=R(1)\alpha_{n-1}-R(2)\alpha_{n-2}
+\left|
\begin{array}{ccccc}
R(3)&1&&&\\
R(4)&R(1)&&&\\
\vdots&\vdots&\ddots&1&\\
R(n-1)&R(n-4)&\cdots&R(1)&1\\
R(n)&R(n-3)&\cdots&R(2)&R(1)
\end{array}
\right|\\
=R(1)\alpha_{n-1}-R(2)\alpha_{n-2}
+\cdots
+(-1)^{n}\left|\begin{array}{cc}
R(n-1)&1\\
R(n)&R(1)
\end{array}\right|\\
=\sum_{j=1}^n(-1)^{j-1}R(j)\alpha_{n-j}=\alpha_n\,.
\end{multline*}
On the contrary, if (\ref{10a}) is true, by expanding the determinant, we get the relation (\ref{10b}).
\end{proof}

Fubini numbers satisfy the required relation.  Therefore, we have a determinant expression of Fubini numbers similarly.

\begin{theorem}\label{teo1}
For $n\ge 1$, we have
$$
F_n=n!\left|
\begin{array}{ccccc}
\frac{1}{1!}&1&&&\\
-\frac{1}{2!}&\frac{1}{1!}&&&\\
\vdots&\vdots&\ddots&1&\\
\frac{(-1)^{n-2}}{(n-1)!}&\frac{(-1)^{n-3}}{(n-2)!}&\cdots&\frac{1}{1!}&1\\
\frac{(-1)^{n-1}}{n!}&\frac{(-1)^{n-2}}{(n-1)!}&\cdots&-\frac{1}{2!}&\frac{1}{1!}
\end{array}
\right|\,.
$$
\label{fubini:det}
\end{theorem}
\begin{proof}
From the relation (\ref{fubini:rec}),
we apply Theorem \ref{det} with
$$
\alpha_n=\frac{F_n}{n!}\quad\hbox{and}\quad R(j)=\frac{(-1)^{j-1}}{j!
}\,.
$$
\end{proof}

\noindent
{\bf Examples.}
Since it is well-known that
$$
\sum_{m=0}^n\binom{n+1}{m}B_m=0\,,
$$
we get
$$
B_n=-\sum_{j=1}^n\frac{-n!}{(j+1)!(n-j)!}B_{n-j}\,.
$$
Applying Theorem \ref{det} with
$$
\alpha_n=(-1)^n\frac{B_n}{n!}\quad\hbox{and}\quad R(j)=\frac{1}{(j+1)!
}\,,
$$
we have the determinant expression (\ref{ber:det}).

Since
\[\frac{c_n}{n!}=\sum_{k=0}^{n-1}\frac{(-1)^{n-k+1}}{n-k+1}\frac{c_k}{k!}=\sum_{j=1}^n\frac{(-1)^{j+1}}{j+1}\frac{c_{n-j}}{(n-j)!}\]
(\cite[Theorem 2.1]{MSV}),
applying Theorem \ref{det} with
$$
\alpha_n=\frac{c_n}{n!}\quad\hbox{and}\quad R(j)=\frac{1}{j+1}\,,
$$
we have the determinant expression (\ref{cau:det}).

Since it is well-known that
$$
\sum_{m=0}^n\binom{2 n}{2 m}E_{2 m}=0\,,
$$
we get
$$
(-1)^n\frac{E_{2 n}}{(2 n)!}=\sum_{j=1}^n\frac{(-1)^{j-1}}{(2 j)!}\frac{(-1)^{n-j}E_{2(n-j)}}{(2 n-2 j)!}\,.
$$
Applying Theorem \ref{det} with
$$
\alpha_n=(-1)^n\frac{E_{2 n}}{(2 n)!}\quad\hbox{and}\quad R(j)=\frac{1}{(2 j)!
}\,,
$$
we have the determinant expression (\ref{eul:det}).

\subsection{Applications of the Trudi's Formula}

We can use the Trudi's formula to obtain an explicit formula for the numbers $\alpha_n$ in Theorem \ref{propo1}.  This relation is known as Trudi's formula \cite[Vol.3, p.214]{Muir},\cite{Trudi} and the case $a_0=1$ of this formula is known as Brioschi's formula \cite{Brioschi},\cite[Vol.3, pp.208--209]{Muir}.

\begin{theorem}[Trudi's formula \cite{Trudi}]
Let $m$ be a positive integer. Then
 \begin{multline}
\begin{vmatrix}
a_1  & a_2   &  \cdots   & & a_m  \\
a_{0}  & a_{1}    &  \cdots  & &   \\
\vdots  &  \vdots &  \ddots  &  & \vdots  \\
0  & 0    &  \cdots  &a_1  & a_{2}  \\
0  & 0   &  \cdots  & a_0  & a_1
\end{vmatrix}\\
=
\sum_{t_1 + 2t_2 + \cdots +mt_m=m}\binom{t_1+\cdots + t_m}{t_1, \dots,t_m}(-a_0)^{m-t_1-\cdots - t_m}a_1^{t_1}a_2^{t_2}\cdots a_m^{t_m}, \label{trudi}
\end{multline}
where $\binom{t_1+\cdots + t_m}{t_1, \dots,t_m}$ is the multinomial coefficient.
\end{theorem}

From Trudi's formula,  it is possible to give the combinatorial formula
$$\alpha_n=\sum_{t_1 + 2t_2 + \cdots + nt_n=n}\binom{t_1+\cdots + t_n}{t_1, \dots, t_n}(-1)^{n-t_1-\cdots - t_n}R(1)^{t_1}R(2)^{t_2}\cdots R(n)^{t_n}.$$

In addition, there exists the following inversion formula, which follows from Theorem \ref{propo1}.

\begin{Lem}\label{lema}
If $\{\alpha_n\}_{n\geq 0}$ is a sequence defined by $\alpha_0=1$ and
$$\alpha_n=\begin{vmatrix} R(1) & 1 & & \\
R(2) & \ddots &  \ddots & \\
\vdots & \ddots &  \ddots & 1\\
R(n) & \cdots &  R(2) & R(1) \\
 \end{vmatrix},  \ \text{then} \ R(n)=\begin{vmatrix} \alpha_1 & 1 & & \\
\alpha_2 & \ddots &  \ddots & \\
\vdots & \ddots &  \ddots & 1\\
\alpha_n & \cdots &  \alpha_2 & \alpha_1 \\
 \end{vmatrix}.$$
 Moreover, if
 $$A=\begin{bmatrix}
  1 &  & & \\
\alpha_1 & 1  &   & \\
\vdots & \ddots &  \ddots & \\
\alpha_n& \cdots &  \alpha_1 & 1 \\
 \end{bmatrix}, \ \text{then} \  A^{-1}=\begin{bmatrix}
  1 &  & & \\
R(1) & 1  &   & \\
\vdots & \ddots &  \ddots & \\
R(n) & \cdots &  R(1) & 1 \\
 \end{bmatrix}.$$
\end{Lem}

From above relations and Theorem \ref{teo1} we obtain an explicit formula for the Fubini sequence.

\begin{Cor}\label{coro1}
For $n\geq 1$
\begin{multline*}
\frac{F_{n}}{n!}=\sum_{t_1 + 2t_2 + \cdots + nt_n=n}\binom{t_1+\cdots + t_n}{t_1, \dots,t_n}(-1)^{n-t_1-\cdots - t_n}\\ \times\left(\frac{1}{1!}\right)^{t_1}\left(\frac{(-1)^1}{2!}\right)^{t_2}\cdots \left(\frac{(-1)^{n-1}}{n!}\right)^{t_n}.
\end{multline*}
Moreover,
$$\begin{vmatrix} \frac{F_1}{1!} & 1 & & \\
\frac{F_2}{2!} & \ddots &  \ddots & \\
\vdots & \ddots &  \ddots & 1\\
\frac{F_n}{n!} & \cdots &  \frac{F_2}{2!} & \frac{F_1}{1!} \\
 \end{vmatrix}=\frac{(-1)^{n-1}}{n!},$$
 and
 $$\begin{bmatrix} 1 &  &  & & \\
\frac{F_1}{1!} & 1 &   &  & \\
\frac{F_2}{2!} & \frac{F_1}{1!}  &  1  &  & \\
\vdots &  &  \ddots &  & \\
\frac{F_n}{n!} & \cdots &  \frac{F_2}{2!} & \frac{F_1}{1!} &1 \\
 \end{bmatrix}^{-1}=\begin{bmatrix} 1 &  &  & & \\
\frac{1}{1!} & 1 &   &  & \\
-\frac{1}{2!} & \frac{1}{1!}  &  1  &  & \\
\vdots &  &  \ddots &  & \\
\frac{(-1)^{n-1}}{n!} & \cdots &  -\frac{1}{2!} & \frac{1}{1!} &1 \\
 \end{bmatrix}.$$
\end{Cor}

\noindent
{\bf Example.} For $n=4$, we have the integer partitions 1+1+1+1, 1+1+2, 1+3, 2+2 and 4. Then
\begin{multline*}
F_4=4!\left((-1)^{4-4}\binom{4+0+0+0}{4,0,0,0}\left(\frac{1}{1!}\right)^4  \right. \\ +
  (-1)^{4-2-1}\binom{2+1+0+0}{2,1,0,0}\left(\frac{1}{1!}\right)^2\left(-\frac{1}{2!}\right)  \\
  + (-1)^{4-1-1}\binom{1+0+1+0}{1,0,1,0}\left(\frac{1}{1!}\right)\left(\frac{1}{3!}\right) \\
  + \left. (-1)^{4-2}\binom{0+2+0+0}{0,2,0,0}\left(-\frac{1}{2!}\right)^2  + (-1)^{4-1}\binom{0+0+0+1}{0,0,0,1}\left(-\frac{1}{4!}\right)\right)\\
  =4!\left( 1 + \frac32 + \frac 13 + \frac14 + \frac{1}{24}\right)=75.\end{multline*}

\section{Restricted Fubini numbers}

The {\it restricted Fubini numbers} are defined by
$$
F_{n,\le m}=\sum_{k=0}^n k!\sts{n}{k}_{\le m}\,.
$$
They satisfy the recurrence relation:
\begin{equation}
F_{n,\le m}=\sum_{j=1}^m\binom{n}{j}F_{n-j,\le m}\,.
\label{resfubini:rec}
\end{equation}

\begin{theorem}\label{thm_rFegenf}
The (exponential) generating function of restricted Fubini numbers is given by
$$
\frac{1}{1-t-\frac{t^2}{2!}-\cdots-\frac{t^m}{m!}}=\sum_{n=0}^\infty F_{n,\le m}\frac{t^n}{n!}\,.
$$
\label{r-fubini:gf}
\end{theorem}
\begin{proof}
Since the generating function of the restricted Stirling numbers of the second kind is given by
$$
\frac{1}{k!}\left(t+\frac{t^2}{2}+\cdots+\frac{t^m}{m!}\right)^k=\sum_{n=k}^{m k}\sts{n}{k}_{\le m}\frac{t^n}{n!}\,,
$$
we have
\begin{align*}
\sum_{n=0}^\infty F_{n,\le m}\frac{t^n}{n!}&=\sum_{n=0}^\infty\sum_{k=0}^n k!\sts{n}{k}_{\le m}\frac{t^n}{n!}=\sum_{k=0}^\infty k!\sum_{n=k}^\infty\sts{n}{k}_{\le m}\frac{t^n}{n!}\\
&=\sum_{k=0}^\infty k!\frac{1}{k!}\left(t+\frac{t^2}{2}+\cdots+\frac{t^m}{m!}\right)^k=\frac{1}{1-t-\frac{t^2}{2}-\cdots-\frac{t^m}{m!}}\,.
\end{align*}
\end{proof}

\begin{theorem}
For $n\ge 1$,
$$
F_{n,\le m}=n!\left|
\begin{array}{ccccccc}
\frac{1}{1!}&1&0&&&&0\\
-\frac{1}{2!}&\frac{1}{1!}&1&&&\\
\vdots&\vdots&\vdots&\ddots&&&\\
\frac{(-1)^{m-1}}{m!}&\frac{(-1)^{m-2}}{(m-1)!}&\frac{(-1)^{m-3}}{(m-2)!}&\cdots&&&\\
0&\frac{(-1)^{m-1}}{m!}&&&&1&0\\
\vdots&&&&&\frac{1}{1!}&1\\
0&\cdots&0&\frac{(-1)^{m-1}}{m!}&\cdots&-\frac{1}{2!}&\frac{1}{1!}
\end{array}
\right|\,.
$$
Namely, all the values $\frac{(-1)^{j-1}}{j!}$ ($m+1\le j\le n$) in the determinant in Theorem \ref{fubini:det} are replaced by $0$.
\label{r-fubini:det}
\end{theorem}

\noindent
{\it Remark.}
When $m\ge n$, the determinant in Theorem \ref{r-fubini:det} matches that in Theorem \ref{fubini:det}.   In fact, $F_{n,\le m}=F_n$ when $m\ge n$.

\begin{proof}[Proof of Theorem \ref{r-fubini:det}]
The result is trivial for $n=1$, because $F_{1,\le m}=1$ ($m\ge 1$).
Assume that the result is true, up to $n-1$.  Then, by expanding the determinant at the first row,
\begin{align*}
&\left|
\begin{array}{ccccccc}
\frac{1}{1!}&1&0&&&&0\\
-\frac{1}{2!}&\frac{1}{1!}&1&&&\\
\vdots&\vdots&\vdots&\ddots&&&\\
\frac{(-1)^{m-1}}{m!}&\frac{(-1)^{m-2}}{(m-1)!}&\frac{(-1)^{m-3}}{(m-2)!}&\cdots&&&\\
0&\frac{(-1)^{m-1}}{m!}&&&&1&0\\
\vdots&&&&&\frac{1}{1!}&1\\
0&\cdots&0&\frac{(-1)^{m-1}}{m!}&\cdots&-\frac{1}{2!}&\frac{1}{1!}
\end{array}
\right|\\
&=\frac{1}{1!}\frac{F_{n-1,\le m}}{(n-1)!}
-\left|
\begin{array}{ccccccc}
-\frac{1}{2!}&1&0&&&&0\\
\frac{1}{3!}&\frac{1}{1!}&1&&&\\
\vdots&\vdots&\vdots&\ddots&&&\\
\frac{(-1)^{m-1}}{m!}&\frac{(-1)^{m-2}}{(m-1)!}&\frac{(-1)^{m-3}}{(m-2)!}&\cdots&&&\\
0&\frac{(-1)^{m-1}}{m!}&&&&1&0\\
\vdots&&&&&\frac{1}{1!}&1\\
0&\cdots&0&\frac{(-1)^{m-1}}{m!}&\cdots&-\frac{1}{2!}&\frac{1}{1!}
\end{array}
\right|\\
&=\frac{1}{1!}\frac{F_{n-1,\le m}}{(n-1)!}+\frac{1}{2!}\frac{F_{n-2,\le m}}{(n-2)!}+\cdots\\
&\quad +\frac{1}{m!}\frac{F_{n-m,\le m}}{(n-m)!}
-\left|
\begin{array}{ccccccc}
0&1&0&&&&0\\
0&\frac{1}{1!}&1&&&\\
\vdots&\vdots&\vdots&\ddots&&&\\
0&\frac{(-1)^{m-2}}{(m-1)!}&\frac{(-1)^{m-3}}{(m-2)!}&\cdots&&&\\
0&\frac{(-1)^{m-1}}{m!}&&&&1&0\\
\vdots&&&&&\frac{1}{1!}&1\\
0&\cdots&0&\frac{(-1)^{m-1}}{m!}&\cdots&-\frac{1}{2!}&\frac{1}{1!}
\end{array}
\right|\\
&=\sum_{j=1}^m\frac{1}{j!}\frac{F_{n-j,\le m}}{(n-j)!}=\frac{F_{n,\le m}}{n!}\,.
\end{align*}
Here, we used the relation (\ref{resfubini:rec}).
\end{proof}

\noindent
{\bf Example.}
By Theorem \ref{fubini:det},
$$
6!\left|\begin{array}{ccccc}
1&1&0&0&0\\
-\frac{1}{2}&1&1&0&0\\
\frac{1}{6}&-\frac{1}{2}&1&1&0\\
-\frac{1}{24}&\frac{1}{6}&-\frac{1}{2}&1&1\\
\frac{1}{120}&-\frac{1}{24}&\frac{1}{6}&-\frac{1}{2}&1
\end{array}\right|
=541
$$
and
$$
5!\left|\begin{array}{cccccc}
1&1&0&0&0\\
-\frac{1}{2}&1&1&0&0&0\\
\frac{1}{6}&-\frac{1}{2}&1&1&0&0\\
-\frac{1}{24}&\frac{1}{6}&-\frac{1}{2}&1&1&0\\
\frac{1}{120}&-\frac{1}{24}&\frac{1}{6}&-\frac{1}{2}&1&1\\
-\frac{1}{720}&\frac{1}{120}&-\frac{1}{24}&\frac{1}{6}&-\frac{1}{2}&1
\end{array}\right|
=4683\,.
$$
Since
\[\frac{1}{2-e^t}
=1+t+\frac{3 t^2}{2}+\frac{13 t^3}{6}+\frac{25 t^4}{8}
+\frac{541 t^5}{120}+\frac{1561 t^6}{240}+\frac{47293 t^7}{5040}+\frac{36389 t^8}{2688}+\cdots\,\]
we also have
$$
F_5=\frac{541}{120}5!=541\quad\hbox{and}\quad F_6=\frac{1561}{240}6!=4683\,.
$$
On the other hand, by Theorem \ref{r-fubini:det},
$$
5!\left|\begin{array}{ccccc}
1&1&0&0&0\\
-\frac{1}{2}&1&1&0&0\\
\frac{1}{6}&-\frac{1}{2}&1&1&0\\
0&\frac{1}{6}&-\frac{1}{2}&1&1\\
0&0&\frac{1}{6}&-\frac{1}{2}&1
\end{array}\right|
=530
$$
and
$$
6!\left|\begin{array}{cccccc}
1&1&0&0&0\\
-\frac{1}{2}&1&1&0&0&0\\
\frac{1}{6}&-\frac{1}{2}&1&1&0&0\\
0&\frac{1}{6}&-\frac{1}{2}&1&1&0\\
0&0&\frac{1}{6}&-\frac{1}{2}&1&1\\
0&0&0&\frac{1}{6}&-\frac{1}{2}&1
\end{array}\right|
=4550\,.
$$
Since
\[\frac{1}{1-(t+\frac{t^2}{2}+\frac{t^3}{6})}
=1+t+\frac{3 t^2}{2}+\frac{13 t^3}{6}+\frac{37 t^4}{12}+\frac{53 t^5}{12}+\frac{455 t^6}{72}+\frac{217 t^7}{24}+\frac{207 t^8}{16}+\cdots\,\]
we also have
$$
F_{5,\le 3}=\frac{53}{12}5!=530\quad\hbox{and}\quad F_{6,\le 3}=\frac{455}{72}6!=4550\,.
$$

From Trudi's formula and Lemma \ref{lema} we obtain the following relations.
\begin{Cor}\label{coro2}
For $n\geq 1$
\begin{multline*}
\frac{F_{n, \leq m}}{n!}=\sum_{t_1 + 2t_2 + \cdots + mt_m=n}\binom{t_1+\cdots + t_m}{t_1, \dots, t_m}\\ \times(-1)^{n-t_1-\cdots - t_m}\left(\frac{1}{1!}\right)^{t_1}\left(\frac{(-1)^1}{2!}\right)^{t_2}\cdots \left(\frac{(-1)^{m-1}}{m!}\right)^{t_m}.
\end{multline*}
Moreover,
$$\begin{vmatrix} \frac{F_{1, \leq m}}{1!} & 1 & & \\
\frac{F_{2, \leq m}}{2!} & \ddots &  \ddots & \\
\vdots & \ddots &  \ddots & 1\\
\frac{F_{n, \leq m}}{n!} & \ddots &  \frac{F_{2, \leq m}}{2!} & \frac{F_{1, \leq m}}{1!} \\
 \end{vmatrix}=0,$$
and
 \begin{multline*}
 \begin{bmatrix} 1 &  &  & & \\
\frac{F_{1, \leq m}}{1!} & 1 &   &  & \\
\frac{F_{2, \leq m}}{2!} & \frac{F_{1, \leq m}}{1!}  &  1  &  & \\
\vdots &  &  \ddots &  & \\
\frac{F_{n, \leq m}}{n!} & \cdots &  \frac{F_{2, \leq m}}{2!} & \frac{F_{1, \leq m}}{1!} &1 \\
 \end{bmatrix}^{-1}\\
 =\begin{bmatrix} 1 &  &  & & \\
\frac{1}{1!} & 1 &   &  & &&\\
-\frac{1}{2!} & \frac{1}{1!}  &  1  & &&  & \\
\vdots &  &  \ddots &  &  &&\\
\frac{(-1)^{m-1}}{m!} & \frac{(-1)^{m-2}}{(m-1)!} & \frac{(-1)^{m-3}}{(m-2)!} &   \cdots &  &  & \\
0 &\frac{(-1)^{m-1}}{m!} & &   \cdots &  & 1 & \\
\vdots & & &   \cdots &  & \frac{1}{1!} &1 \\
0 & &  &   \cdots & \frac{(-1)^{m-1}}{m!} &\cdots &\frac{1}{1!}& 1 \\
 \end{bmatrix}.
 \end{multline*}
\end{Cor}

\section{Associated Fubini numbers}

The {\it associated Fubini numbers} are defined by
$$
F_{n,\ge m}=\sum_{k=0}^n k!\sts{n}{k}_{\ge m}\,.
$$
They satisfy the recurrence relation:
\begin{equation}
F_{n,\ge m}=\sum_{j=m}^n\binom{n}{j}F_{n-j,\ge m}\,.
\label{assofubini:rec}
\end{equation}

\begin{theorem}
The (exponential) generating function of restricted Fubini numbers is given by
$$
\frac{1}{1-\frac{t^m}{m!}-\frac{t^{m+1}}{(m+1)!}-\cdots}=\sum_{n=0}^\infty F_{n,\ge m}\frac{t^n}{n!}\,.
$$
\label{a-fubini:gf}
\end{theorem}
\begin{proof}
Since the generating function of the associated Stirling numbers of the second kind is given by
$$
\frac{1}{k!}\left(\frac{t^m}{m!}+\frac{t^{m+1}}{(m+1)!}+\cdots\right)^k=\sum_{n=k}^{m k}\sts{n}{k}_{\ge m}\frac{t^n}{n!}\,,
$$
we have
\begin{align*}
\sum_{n=0}^\infty F_{n,\ge m}\frac{t^n}{n!}
&=\sum_{n=0}^\infty\sum_{k=0}^n k!\sts{n}{k}_{\ge m}\frac{t^n}{n!}=\sum_{k=0}^\infty k!\sum_{n=k}^\infty\sts{n}{k}_{\ge m}\frac{t^n}{n!}\\
&=\sum_{k=0}^\infty k!\frac{1}{k!}\left(\frac{t^m}{m!}+\frac{t^{m+1}}{(m+1)!}+\cdots\right)^k=\frac{1}{1-\frac{t^m}{m!}-\frac{t^{m+1}}{(m+1)!}-\cdots}\,.
\end{align*}
\end{proof}

\begin{theorem}
For $n\ge m\ge 1$,
$$
F_{n,\ge m}=n!\left|
\begin{array}{ccccccc}
0&1&0&\cdots&&&0\\
\vdots&0&1&0&\cdots&&\vdots\\
0&\vdots&\ddots&\ddots&&&\\
\frac{(-1)^{m-1}}{m!}&0&&&&&\vdots\\
\vdots&\ddots&&&&\ddots&0\\
\frac{(-1)^{n-2}}{(n-1)!}&&&\ddots&&0&1\\
\frac{(-1)^{n-1}}{n!}&\frac{(-1)^{n-2}}{(n-1)!}&\cdots&\frac{(-1)^{m-1}}{m!}&0&\cdots&0
\end{array}
\right|\,.
$$
Namely, all the values $\frac{(-1)^{j-1}}{j!}$ ($1\le j\le m-1$) in the determinant in Theorem \ref{fubini:det} are replaced by $0$.
\label{a-fubini:det}
\end{theorem}

\noindent
{\it Remark.}
When $m=1$, the determinant in Theorem \ref{a-fubini:det} matches that in Theorem \ref{fubini:det}.   In fact, $F_{n,\ge 1}=F_n$.

\begin{proof}[Proof of Theorem \ref{a-fubini:det}]
For convenience, put
$$
G_{n,m}:=\left|
\begin{array}{ccccccc}
0&1&0&\cdots&&&0\\
\vdots&0&1&0&\cdots&&\vdots\\
0&\vdots&\ddots&\ddots&&&\\
\frac{(-1)^{m-1}}{m!}&0&&&&&\vdots\\
\vdots&\ddots&&&&\ddots&0\\
\frac{(-1)^{n-2}}{(n-1)!}&&&\ddots&&0&1\\
\frac{(-1)^{n-1}}{n!}&\frac{(-1)^{n-2}}{(n-1)!}&\cdots&\frac{(-1)^{m-1}}{m!}&0&\cdots&0
\end{array}
\right|\,.
$$
Expanding the determinant at the first row of $G_{n,m}$ continually,  if $n<2 m$, then
\begin{align*}
G_{n,m}&=(-1)^{m-1}\left|\begin{array}{cccc}
\frac{(-1)^{m-1}}{(m-1)!}&1&&\\
\vdots&&&\\
\vdots&&&1\\
\frac{(-1)^{n-1}}{n!}&0&\cdots&0
\end{array}\right|\\
&=(-1)^{m-1}(-1)^{n-m}\frac{(-1)^{n-1}}{n!}=\frac{1}{n!}\,. \end{align*}
Since $F_{n,\ge m}=1$ ($n<2 m$), we have $F_{n,\ge m}=n! G_{n,m}$, as desired.

If $n\ge 2 m$, then
\begin{align*}
G_{n,m}&=(-1)^{m-1}
\left|\begin{array}{cccc}
\frac{(-1)^{m-1}}{m!}&1&&\\
\vdots&&&\\
\vdots&&&1\\
\frac{(-1)^{n-1}}{n!}&\frac{(-1)^{n-m+1}}{(n-m)!}&&0
\end{array}\right|\\
&=\frac{1}{m!}\left|\begin{array}{cccccc}
0&1&0&&&\\
\vdots&&&&&\\
0&&&&&\\
\frac{(-1)^{m-1}}{m!}&&&&&0\\
\vdots&&&&&1\\
\frac{(-1)^{n-m+1}}{(n-m)!}&\cdots&\frac{(-1)^{m-1}}{m!}&0&\cdots&0
\end{array}\right|\\
&\quad-(-1)^{m-1}\left|\begin{array}{cccc}
\frac{(-1)^{m}}{(m+1)!}&1&&\\
\vdots&&&\\
\vdots&&&1\\
\frac{(-1)^{n-1}}{n!}&\frac{(-1)^{n-m}}{(n-m-1)!}&&0
\end{array}\right|\\
&=\frac{1}{m!}\frac{F_{n-m,\ge m}}{(n-m)!}+\frac{1}{(m+1)!}\left|\begin{array}{cccccc}
0&1&0&&&\\
\vdots&&&&&\\
0&&&&&\\
\frac{(-1)^{m-1}}{m!}&&&&&0\\
\vdots&&&&&1\\
\frac{(-1)^{n-m}}{(n-m-1)!}&\cdots&\frac{(-1)^{m-1}}{m!}&0&\cdots&0
\end{array}\right|\\
&\quad-(-1)^{m}\left|\begin{array}{cccc}
\frac{(-1)^{m+1}}{(m+2)!}&1&&\\
\vdots&&&\\
\vdots&&&1\\
\frac{(-1)^{n-1}}{n!}&\frac{(-1)^{n-m+1}}{(n-m-2)!}&&0
\end{array}\right|=\sum_{j=m}^n\frac{1}{j!}\frac{F_{n-j,\ge m}}{(n-j)!}=\frac{F_{n,\ge m}}{n!}\,.
\end{align*}
Here, we use the relation (\ref{assofubini:rec}).
\end{proof}

\noindent
{\bf Example.}

By Theorem \ref{a-fubini:det},
$$
5!\left|\begin{array}{ccccc}
0&1&0&0&0\\
0&0&1&0&0\\
\frac{1}{6}&0&0&1&0\\
-\frac{1}{24}&\frac{1}{6}&0&0&1\\
\frac{1}{120}&-\frac{1}{24}&\frac{1}{6}&0&0
\end{array}\right|
=1
$$
and
$$
6!\left|\begin{array}{cccccc}
0&1&0&0&0&0\\
0&0&1&0&0&0\\
\frac{1}{6}&0&0&1&0&0\\
-\frac{1}{24}&\frac{1}{6}&0&0&1&0\\
\frac{1}{120}&-\frac{1}{24}&\frac{1}{6}&0&0&1\\
-\frac{1}{720}&\frac{1}{120}&-\frac{1}{24}&\frac{1}{6}&0&0
\end{array}\right|
=21\,.
$$
Since
\[\frac{1}{1-(\frac{t^3}{6}+\frac{t^4}{24}+\cdots)}
=1+\frac{t^3}{6}+\frac{t^4}{24}+\frac{t^5}{120}+\frac{7 t^6}{240}
+\frac{61 t^8}{13440}+\frac{2101 t^9}{362880}+\cdots\,,\]
we also have
$$
F_{5,\ge 3}=\frac{1}{6}6!=1\quad\hbox{and}\quad F_{6,\ge 3}=\frac{7}{240}6!=21\,.
$$

From Trudi's formula and Lemma \ref{lema} we obtain the following relations.
\begin{Cor}\label{coro3}
For $n\geq 1$
\begin{multline}\frac{F_{n, \geq m}}{n!}=\sum_{mt_m + (m+1)t_{m+1} + \cdots + nt_n=n}\binom{t_m+\cdots + t_n}{t_m, \dots,t_n}(-1)^{n - t_m - \cdots - t_n}\\\times \left(\frac{(-1)^{m-1}}{m!}\right)^{t_m}\left(\frac{(-1)^{m}}{(m+1)!}\right)^{t_{m+1}}\cdots \left(\frac{(-1)^{n-1}}{n!}\right)^{t_n}.
\end{multline}
Moreover,
$$\begin{vmatrix} \frac{F_{1, \geq m}}{1!} & 1 & & \\
\frac{F_{2, \geq m}}{2!} & \ddots &  \ddots & \\
\vdots & \ddots &  \ddots & 1\\
\frac{F_{n, \geq m}}{n!} & \ddots &  \frac{F_{2, \geq m}}{2!} & \frac{F_{1, \geq m}}{1!} \\
 \end{vmatrix}=\frac{(-1)^{n-1}}{n!},$$
and
 \begin{multline*}
 \begin{bmatrix} 1 &  &  & & \\
\frac{F_{1, \geq m}}{1!} & 1 &   &  & \\
\frac{F_{2, \geq m}}{2!} & \frac{F_{1, \geq m}}{1!}  &  1  &  & \\
\vdots &  &  \ddots &  & \\
\frac{F_{n, \geq m}}{n!} & \cdots &  \frac{F_{2, \geq m}}{2!} & \frac{F_{1, \geq m}}{1!} &1 \\
 \end{bmatrix}^{-1}\\=\begin{bmatrix} 1 &  &  & & \\
0 & 1 &   &  & &&\\
 & 0  &  1  & &&  & \\
\vdots &  &  \ddots &  &  &&\\
\frac{(-1)^{m-1}}{m!} & 0 &  &   \cdots &  &  & \vdots \\
\vdots & \ddots & &    &  & \ddots & 0\\
\frac{(-1)^{n-2}}{(n-1)!} & & &   \cdots &  & 0&1 \\
\frac{(-1)^{n-1}}{n!}  &  \frac{(-1)^{n-2}}{(n-1)!}  & \cdots &  \frac{(-1)^{m-1}}{m!}   & 0&\cdots & 0 & 1 \\
 \end{bmatrix}.
 \end{multline*}
\end{Cor}

\section{Modified incomplete Bernoulli and Cauchy numbers}

In \cite{Ko5,KMS}, restricted Cauchy numbers $c_{n,\le m}$ and associated Cauchy  numbers $c_{n,\ge m}$ are introduced as
\begin{equation}
\sum_{n=0}^\infty c_{n,\le m}\frac{t^n}{n!}=\frac{e^{F_m(t)}-1}{F_m(t)}
\label{r-c:gen}
\end{equation}
and
\begin{equation}
\sum_{n=0}^\infty c_{n,\ge m}\frac{t^n}{n!}=\frac{e^{\ln(1+t)-F_{m-1}(t)}-1}{\ln(1+t)-F_{m-1}(t)}\,,
\label{a-c:gen}
\end{equation}
respectively, where
$$
F_m(t)=t-\frac{t^2}{2}+\cdots+(-1)^{m-1}\frac{t^m}{m}\quad(m\ge 1)\,.
$$
Both incomplete Cauchy numbers are natural extensions of the original Cauchy numbers $c_n$ because they have expressions:
$$
 c_{n,\le m}=\sum_{k=0}^n \stf{n}{k}_{\le m}   \frac{(-1)^{n-k}}{k+1}
$$
and
$$
c_{n,\ge m}=\sum_{k=0}^n\stf{n}{k}_{\geq m}\frac{(-1)^{n-k}}{k+1}
$$
with $c_n=c_{n,\le\infty}=c_{n,\ge 1}$.
However, such incomplete Cauchy numbers cannot be natural extensions of the original one in terms of the determinants, though Fubini numbers do as seen in Theorems \ref{fubini:det}, \ref{r-fubini:det} and \ref{a-fubini:det}.
Therefore, we introduced {\it modified} incomplete Cauchy numbers.  The {\it modified restricted Cauchy numbers} $c_{n,\le m}^\ast$ are defined by
\begin{equation}
\frac{t}{F_m(t)}=\sum_{n=0}^\infty c_{n,\le m}^\ast\frac{t^n}{n!}\quad(m\ge 2)
\label{m-r-c:gen}
\end{equation}
instead of (\ref{r-c:gen}).
The {\it modified associated Cauchy numbers} $c_{n,\ge m}^\ast$ are defined by
\begin{equation}
\frac{t}{\ln(1+t)-F_{m-1}(t)+t}=\sum_{n=0}^\infty c_{n,\ge m}^\ast\frac{t^n}{n!}\quad(m\ge 2)
\label{m-a-c:gen}
\end{equation}
instead of (\ref{a-c:gen}).
Note that $c_n=c_{n,\le m}^\ast$ if $n\le m-1$, and $c_n=c_{n,\ge 2}^\ast$.

\begin{theorem}
For integers $n\ge 1$ and $m\ge 2$,
$$
c_{n,\le m}^\ast=n!\left|\begin{array}{cccccc}
\frac{1}{2}&1&0&&&\\
\vdots&\ddots&\ddots&\ddots&&\\
\frac{1}{m}&&&&&\\
0&\ddots&&&&0\\
&&&&&1\\
&&0&\frac{1}{m}&\cdots&\frac{1}{2}
\end{array}\right|\,.
$$
\label{m-r-c:det}
\end{theorem}

\begin{theorem}
For integers $n$ and $m$ with $n-1\ge m\ge 2$,
$$
c_{n,\ge m}^\ast=n!\left|\begin{array}{cccccc}
0&1&0&&&\\
\vdots&\ddots&\ddots&\ddots&&\\
\frac{1}{m}&&&&\ddots&\\
\frac{1}{m+1}&\ddots&&&\ddots&0\\
\vdots&&\ddots&&\ddots&1\\
\frac{1}{n+1}&\cdots&\frac{1}{m+1}&\frac{1}{m}&\cdots&0
\end{array}\right|\,.
$$
\label{m-a-c:det}
\end{theorem}

\noindent
{\it Remark.}
If $n\le m-1$, this result is reduced to (\ref{cau:det}).

\begin{proof}[Proof of Theorem \ref{m-r-c:det}]
By (\ref{m-r-c:gen}),
\[1=\left(\sum_{j=0}^{m-1}(-1)^j\frac{t^j}{j+1}\right)\left(\sum_{l=0}^\infty c_{l,\le m}^\ast\frac{t^l}{l!}\right)=\sum_{n=0}^\infty\sum_{j=0}^{\min\{n,m-1\}}\frac{(-1)^j}{j+1}\frac{c_{n-j,\le m}^\ast}{(n-j)!}t^n\,.\]
Hence,
\begin{equation}
\frac{c_{n,\le m}^\ast}{n!}=\sum_{j=1}^m\frac{(-1)^{j+1}}{j+1}\frac{c_{n-j,\le m}^\ast}{(n-j)!}\quad(n\ge 1)\,.
\label{m-r-c:eq}
\end{equation}
By (\ref{m-r-c:gen}), it is clear that $c_{1,\le m}^\ast=\frac{1}{2}$.
Assume that the result is valid up to $n-1$.  Then, by expanding the determinant at the first row, we have
\begin{align*}
&\left|\begin{array}{cccccc}
\frac{1}{2}&1&0&&&\\
\vdots&\ddots&\ddots&\ddots&&\\
\frac{1}{m}&&&&&\\
0&\ddots&&&&0\\
&&&&&1\\
&&0&\frac{1}{m}&\cdots&\frac{1}{2}
\end{array}\right|=\frac{1}{2}\frac{c_{n-1,\le m}^\ast}{(n-1)!}-\left|\begin{array}{cccccc}
\frac{1}{3}&1&0&&&\\
\vdots&\ddots&\ddots&\ddots&&\\
\frac{1}{m}&&&&&\\
0&\ddots&&&&0\\
&&&&&1\\
&&0&\frac{1}{m}&\cdots&\frac{1}{2}
\end{array}\right|\\
&=\frac{1}{2}\frac{c_{n-1,\le m}^\ast}{(n-1)!}-\frac{1}{3}\frac{c_{n-2,\le m}^\ast}{(n-2)!}+\cdots
+(-1)^{m}\left|\begin{array}{ccccc}
\frac{1}{m}&1&0&&\\
0&\frac{1}{2}&&&\\
\vdots&&&&0\\
\vdots&&&&1\\
0&&&&\frac{1}{2}
\end{array}\right|\\
&=\frac{1}{2}\frac{c_{n-1,\le m}^\ast}{(n-1)!}-\frac{1}{3}\frac{c_{n-2,\le m}^\ast}{(n-2)!}+\cdots
+\frac{(-1)^{m}}{m}\frac{c_{n-m+1,\le m}^\ast}{(n-m+1)!}=\frac{c_{n,\le m}^\ast}{n!}\,.
\end{align*}
Here, we used the relation (\ref{m-r-c:eq}).
\end{proof}

\noindent
{\it Remark.}
If $m=2$,  this result is reduced to (\ref{cau:det}).

\begin{proof}[Proof of Theorem \ref{m-a-c:det}]
If $n+1<2 m$, the identity is equivalent to
$$
\frac{c_{n,\ge m}^\ast}{n!}=(-1)^m\left|\begin{array}{ccccccc}
\frac{1}{m}&1&0&&&&\\
\vdots&0&\ddots&&&&\\
&\vdots&&&&&\\
&0&&&&&\\
&\frac{1}{m}&&&&\ddots&0\\
\vdots&\vdots&\ddots&&&0&1\\
\frac{1}{n+1}&\frac{1}{n-m+2}&\cdots&\frac{1}{m}&0&\cdots&0
\end{array}\right|\,.
$$
If $n+1\ge 2 m$, the identity is equivalent to
$$
\frac{c_{n,\ge m}^\ast}{n!}=(-1)^m\left|\begin{array}{cccc}
\frac{1}{m}&1&0&\\
\vdots&0&&0\\
\vdots&\vdots&\ddots&1\\
\frac{1}{n+1}&0&\cdots&0
\end{array}\right|\,.
$$
By (\ref{m-a-c:gen}), we have
\begin{multline*}
1=\left(1+\sum_{j=m-1}^\infty(-1)^j\frac{t^j}{j+1}\right)\left(\sum_{l=0}^\infty c_{l,\ge m}^\ast\frac{t^l}{l!}\right)\\
=\sum_{n=0}^\infty c_{n,\ge m}^\ast\frac{t^n}{n!}+\sum_{n=0}^\infty\sum_{j=m-1}^n\frac{(-1)^j}{j+1}\frac{c_{n-j,\ge m}^\ast}{(n-j)!}t^n\,.
\end{multline*}
Hence, for $n\ge 1$
$$
\frac{c_{n,\ge m}^\ast}{n!}=\sum_{j=m-1}^n\frac{(-1)^{j+1}}{j+1}\frac{c_{n-j,\ge m}^\ast}{(n-j)!}\,.
$$
By using this relation, we can obtain the result, similarly to the proof of Theorem \ref{m-r-c:det}.
\end{proof}

For the modified Cauchy numbers, we can also give an explicit expression similar to that given in Corollaries \ref{coro1}, \ref{coro2} and \ref{coro3}.
\begin{Cor}
For $n\geq 1$
\begin{multline*}
\frac{c_{n,\leq m}^*}{n!}=\sum_{t_1 + 2t_2 + \cdots + (m-1)t_{m-1}=n}\binom{t_1+\cdots + t_{m-1}}{t_1, \dots, t_{m-1}}\\ \times
(-1)^{n-t_1-\cdots - t_{m-1}}\left(\frac{1}{2}\right)^{t_1}\left(\frac{1}{3}\right)^{t_2}\cdots \left(\frac{1}{m}\right)^{t_{m-1}}.
\end{multline*}
Moreover,
$$\begin{vmatrix} \frac{c_{1, \leq m}^*}{1!} & 1 & & \\
\frac{c_{2, \leq m}^*}{2!} & \ddots &  \ddots & \\
\vdots & \ddots &  \ddots & 1\\
\frac{c_{n, \leq m}^*}{n!} & \ddots &  \frac{c_{2, \leq m}^*}{2!} & \frac{c_{1, \leq m}^*}{1!} \\
 \end{vmatrix}=0,$$
and
 \begin{multline*}
 \begin{bmatrix} 1 &  &  & & \\
\frac{c_{1, \leq m}^*}{1!} & 1 &   &  & \\
\frac{c_{2, \leq m}^*}{2!} & \frac{c_{1, \leq m}^*}{1!}  &  1  &  & \\
\vdots &  &  \ddots &  & \\
\frac{c_{n, \leq m}^*}{n!} & \cdots &  \frac{c_{2, \leq m}^*}{2!} & \frac{c_{1, \leq m}^*}{1!} &1 \\
 \end{bmatrix}^{-1}\\=\begin{bmatrix} 1 &  &  & & \\
\frac{1}{2} & 1 &   &  & &&\\
\frac{1}{3} & \frac{1}{2}  &  1  & &&  & \\
\vdots &  &  \ddots &  &  &&\\
\frac{1}{m} & \frac{1}{m-1} & \frac{1}{m-2} &   \cdots &  &  & \\
0 &\frac{1}{m} & &   \cdots &  & 1 & \\
\vdots & & &   \cdots &  & \frac{1}{2} &1 \\
0 & &  &   \cdots & \frac{1}{m} &\cdots &\frac{1}{2}& 1 \\
 \end{bmatrix}.
 \end{multline*}
\end{Cor}

\begin{Cor}
For $n\geq 1$
\begin{multline}\frac{c_{n, \geq m}^*}{n!}=\sum_{mt_{m-1} + (m+1)t_{m} + \cdots + (n+1)t_n=n}\binom{t_{m-1}+\cdots + t_n}{t_{m-1}, \dots, t_n}\\ \times(-1)^{n - t_{m-1} - \cdots - t_n}\left(\frac{1}{m}\right)^{t_{m-1}}\left(\frac{1}{m+1}\right)^{t_{m}}\cdots \left(\frac{1}{n+1}\right)^{t_n}.
\end{multline}
Moreover,
$$\begin{vmatrix} \frac{c_{1, \geq m}^*}{1!} & 1 & & \\
\frac{c_{2, \geq m}^*}{2!} & \ddots &  \ddots & \\
\vdots & \ddots &  \ddots & 1\\
\frac{c_{n, \geq m}^*}{n!} & \ddots &  \frac{c_{2, \geq m}^*}{2!} & \frac{c_{1, \geq m}^*}{1!} \\
 \end{vmatrix}=\frac{1}{n+1},$$
and
 \begin{multline*}
 \begin{bmatrix} 1 &  &  & & \\
\frac{c_{1, \geq m}^*}{1!} & 1 &   &  & \\
\frac{c_{2, \geq m}^*}{2!} & \frac{c_{1, \geq m}^*}{1!}  &  1  &  & \\
\vdots &  &  \ddots &  & \\
\frac{c_{n, \geq m}^*}{n!} & \cdots &  \frac{c_{2, \geq m}^*}{2!} & \frac{c_{1, \geq m}^*}{1!} &1 \\
 \end{bmatrix}^{-1}\\=\begin{bmatrix} 1 &  &  & & \\
0 & 1 &   &  & &&\\
 & 0  &  1  & &&  & \\
\vdots &  & &  \ddots  &  &&\\
\frac{1}{m} & 0 &  &   \cdots &  &  & \vdots \\
\vdots & \ddots & &    &  & \ddots & \\
\frac{1}{n} & & &   \cdots &  & 0&1 \\
\frac{1}{n+1}  &  \frac{1}{n}  & \cdots &  \frac{1}{m}   & 0&\cdots & 0 & 1 \\
 \end{bmatrix}.
 \end{multline*}
\end{Cor}

In \cite{KLM}, restricted Bernoulli numbers $B_{n,\le m}$ and associated Bernoulli numbers $B_{n,\ge m}$ are introduced.  However, similarly to incomplete Cauchy numbers, these incomplete Bernoulli numbers must be modified to have determinant expressions.

Define {\it modified restricted Bernoulli numbers} $B_{n,\le m}^\ast$ by
\begin{equation}
\frac{t}{E_m(t)-1}=\sum_{n=0}^\infty B_{n,\le m}^\ast\frac{t^n}{n!}
\label{m-r-b:gen}
\end{equation}
and define
{\it modified associated Bernoulli numbers} $B_{n,\ge m}^\ast$ by
\begin{equation}
\frac{t}{e^t-E_m(t)-1+t}=\sum_{n=0}^\infty B_{n,\ge m}^\ast\frac{t^n}{n!}\,,
\label{m-a-b:gen}
\end{equation}
where
$$
E_m(t)=1+t+\frac{t^2}{2!}+\cdots+\frac{t^m}{m!}\,.
$$

Then these modified incomplete Bernoulli numbers have determinant expressions.  The proofs are similar and omitted.

\begin{theorem}
For integers $n\ge 1$ and $m\ge 2$,
$$
B_{n,\le m}^\ast=(-1)^n n!\left|\begin{array}{cccccc}
\frac{1}{2!}&1&0&&&\\
\vdots&\ddots&\ddots&\ddots&&\\
\frac{1}{m!}&&&&&\\
0&\ddots&&&&0\\
&&&&&1\\
&&0&\frac{1}{m!}&\cdots&\frac{1}{2!}
\end{array}\right|\,.
$$
\label{m-r-b:det}
\end{theorem}

\noindent
{\it Remark.}
If $n\le m-1$, this result is reduced to (\ref{ber:det}).

\begin{theorem}
For integers $n$ and $m$ with $n-1\ge m\ge 2$,
$$
B_{n,\ge m}^\ast=(-1)^nn!\left|\begin{array}{cccccc}
0&1&0&&&\\
\vdots&\ddots&\ddots&\ddots&&\\
\frac{1}{m!}&&&&\ddots&\\
\frac{1}{(m+1)!}&\ddots&&&\ddots&0\\
\vdots&&\ddots&&\ddots&1\\
\frac{1}{(n+1)!}&\cdots&\frac{1}{(m+1)!}&\frac{1}{m!}&\cdots&0
\end{array}\right|\,.
$$
\label{m-a-b:det}
\end{theorem}

\noindent
{\it Remark.}
If $m=2$, this result is reduced to (\ref{ber:det}).

\begin{Cor}
For $n\geq 1$
\begin{multline*}
\frac{(-1)^nB_{n,\leq m}^*}{n!}=\sum_{t_1 + 2t_2 + \cdots + (m-1)t_{m-1}=n}\binom{t_1+\cdots + t_{m-1}}{t_1, \dots, t_{m-1}}\\ \times (-1)^{n-t_1-\cdots - t_{m-1}} \left(\frac{1}{2!}\right)^{t_1}\left(\frac{1}{3!}\right)^{t_2}\cdots \left(\frac{1}{m!}\right)^{t_{m-1}}.
\end{multline*}
Moreover,
$$\begin{vmatrix} \frac{-B_{1, \leq m}^*}{1!} & 1 & & \\
\frac{(-1)^2B_{2, \leq m}^*}{2!} & \ddots &  \ddots & \\
\vdots & \ddots &  \ddots & 1\\
\frac{(-1)^nB_{n, \leq m}^*}{n!} & \ddots &  \frac{(-1)^2B_{2, \leq m}^*}{2!} & \frac{-B_{1, \leq m}^*}{1!} \\
 \end{vmatrix}=0,$$
and
 \begin{multline*}
 \begin{bmatrix} 1 &  &  & & \\
\frac{-B_{1, \leq m}^*}{1!} & 1 &   &  & \\
\frac{(-1)^2B_{2, \leq m}^*}{2!} & \frac{-B_{1, \leq m}^*}{1!}  &  1  &  & \\
\vdots &  &  \ddots &  & \\
\frac{(-1)^nB_{n, \leq m}^*}{n!} & \cdots &  \frac{(-1)^2B_{2, \leq m}^*}{2!} & \frac{-B_{1, \leq m}^*}{1!} &1 \\
 \end{bmatrix}^{-1}\\
 =\begin{bmatrix} 1 &  &  & & \\
\frac{1}{2!} & 1 &   &  & &&\\
\frac{1}{3!} & \frac{1}{2!}  &  1  & &&  & \\
\vdots &  &  \ddots &  &  &&\\
\frac{1}{m!} & \frac{1}{(m-1)!} & \frac{1}{(m-2)!} &   \cdots &  &  & \\
0 &\frac{1}{m!} & &   \cdots &  & 1 & \\
\vdots & & &   \cdots &  & \frac{1}{2!} &1 \\
0 & &  &   \cdots & \frac{1}{m!} &\cdots &\frac{1}{2!}& 1 \\
 \end{bmatrix}.
 \end{multline*}
\end{Cor}

\begin{Cor}
For $n\geq 1$
\begin{multline}\frac{(-1)^nB_{n, \geq m}^*}{n!}=\sum_{mt_{m-1} + (m+1)t_{m} + \cdots + (n+1)t_n=n}\binom{t_{m-1}+\cdots + t_n}{t_{m-1}, \dots, t_n}\\ \times (-1)^{n - t_{m-1} - \cdots - t_n}\left(\frac{1}{m!}\right)^{t_{m-1}}\left(\frac{1}{(m+1)!}\right)^{t_{m}}\cdots \left(\frac{1}{(n+1)!}\right)^{t_n}.
\end{multline}
Moreover,
$$\begin{vmatrix} \frac{-B_{1, \geq m}^*}{1!} & 1 & & \\
\frac{(-1)^2B_{2, \geq m}^*}{2!} & \ddots &  \ddots & \\
\vdots & \ddots &  \ddots & 1\\
\frac{(-1)^nB_{n, \geq m}^*}{n!} & \ddots &  \frac{(-1)^2B_{2, \geq m}^*}{2!} & \frac{-B_{1, \geq m}^*}{1!} \\
 \end{vmatrix}=\frac{1}{(n+1)!},$$
and
 \begin{multline*}
 \begin{bmatrix} 1 &  &  & & \\
\frac{-B_{1, \geq m}^*}{1!} & 1 &   &  & \\
\frac{(-1)^2B_{2, \geq m}^*}{2!} & \frac{-B_{1, \geq m}^*}{1!}  &  1  &  & \\
\vdots &  &  \ddots &  & \\
\frac{(-1)^nB_{n, \geq m}^*}{n!} & \cdots &  \frac{(-1)^2B_{2, \geq m}^*}{2!} & \frac{-B_{1, \geq m}^*}{1!} &1 \\
 \end{bmatrix}^{-1}\\=\begin{bmatrix} 1 &  &  & & \\
0 & 1 &   &  & &&\\
 & 0  &  1  & &&  & \\
\vdots &  & &  \ddots  &  &&\\
\frac{1}{m!} & 0 &  &   \cdots &  &  & \vdots \\
\vdots & \ddots & &    &  & \ddots & \\
\frac{1}{n!} & & &   \cdots &  & 0&1 \\
\frac{1}{(n+1)!}  &  \frac{1}{n!}  & \cdots &  \frac{1}{m!}   & 0&\cdots & 0 & 1 \\
 \end{bmatrix}.
 \end{multline*}
\end{Cor}

\section{The Generalized Fubini Numbers}

The {\em generalized Fubini numbers} $F_{n}^{(k)}$ are defined by
\begin{align*}
\left(\frac{1}{2-e^t}\right)^k=\sum_{j=0}^{\infty}F_{j+1}^{(k)}\frac{t^j}{j!}, \ k\in\mathbb{Z}^+.
\end{align*}

Note that
\begin{align}
\frac{F_{n+1}^{(k)}}{n!}=\sum_{j_1+j_2+\cdots +j_k=n} \frac{F_{j_1+1}}{(j_1+1)!}\frac{F_{j_2+1}}{(j_2+1)!}\cdots \frac{F_{j_k+1}}{(j_k+1)!}.
\end{align}

The generating functions of the generalized Fubini numbers for $k = 2$ and 3 are
\begin{align*}
\left(\frac{1}{2-e^t}\right)^2&= 1+2 t+4 t^2+\frac{22 t^3}{3}+\frac{77 t^4}{6}+\frac{653 t^5}{30}+\frac{6497
   t^6}{180}+\frac{74141 t^7}{1260}+ \cdots\\
\left(\frac{1}{2-e^t}\right)^3&=1+3 t+\frac{15 t^2}{2}+\frac{33 t^3}{2}+\frac{269 t^4}{8}+\frac{2601
   t^5}{40}+\frac{5809 t^6}{48}+\cdots
\end{align*}
\smallskip

Let $T_n=[t_{i,j}]$  be the $n$-square Toeplitz-Hessenberg matrix defined as in Theorem \ref{teo1}, i.e.,
$$t_{i,j}=t_{i-j}=\begin{cases}
\frac{(-1)^{i-1}}{i!}, & \text{if} \ i\geq j; \\
1, & \text{if} \ i+1=j; \\
0, & \text{otherwise.}
\end{cases}$$
Let $A$ and $C$ be matrices of size $n\times n$ and $m\times m$, respectively, and $B$ be a $n\times m$ matrix.
Since
$$\det\begin{bmatrix} A & B \\ 0 & C \end{bmatrix}=\det A \det C,$$
the principal minor $M(i)$ of the matrix $T_n$  is equal to $\frac{F_{i}}{i!}\frac{F_{n-i+1}}{(n-i+1)!}$. It follows that the principal minor  $M(i_1,i_2,\dots,i_l)$ of the matrix $T_n$  is obtained by deleting rows and columns with indices $1\leqslant i_1 < i_2 < \cdots < i_l\leqslant n$:
\begin{align}
M(i_1,i_2,\dots,i_l)=\frac{F_{i_1}}{i_1!}\frac{F_{i_2-i_1}}{(i_2-i_1)!} \cdots \frac{F_{i_l-i_{l-1}}}{(i_l-i_{l-1})!} \frac{F_{n-i_l+1}}{(n-i_l+1)!}.  \label{minorsnara}
\end{align}

Then from (\ref{minorsnara}) we have the following theorem.
\begin{theorem}
Let $S_{n-l}, (l=0, 1, 2, \dots, n-1)$ be the sum of all principal minors of the matrix $T_n$ of order $n-l$. Then
\begin{align}
S_{n-l}=\sum_{j_1+j_2+\cdots + j_{l+1}=n-l} \frac{F_{j_1}+1}{(j_1+1)!}  \frac{F_{j_2}+1}{(j_2+1)!} \cdots  \frac{F_{j_{l+1}}+1}{(j_{l+1}+1)!}.
\end{align}
\end{theorem}

Since the coefficients  of the characteristic polynomial of a matrix are, up to the sign, sums of principal minors of the matrix, then we have the following.

\begin{Cor}
The generalized Fubini number $\frac{F_{n-l+1}^{(l+1)}}{(n-l+1)!}$ is equal, up to the sign, to the coefficient of $x^l$ in the characteristic polynomial $p_{n}(x)$ of the matrix  $T_n$
\end{Cor}

For example,
$$T_5=\begin{bmatrix}
 \frac{1}{1!} & 1 & 0 & 0 & 0 \\
 -\frac{1}{2} &  \frac{1}{1!} & 1 & 0 & 0 \\
 \frac{1}{3!} & -\frac{1}{2!} &  \frac{1}{1!} & 1 & 0 \\
 -\frac{1}{4!} & \frac{1}{3!} & -\frac{1}{2!} &  \frac{1}{1!} & 1 \\
 \frac{1}{5!} & -\frac{1}{4!} & \frac{1}{3!} & -\frac{1}{2!} &  \frac{1}{1!} \\
\end{bmatrix}$$
 then its characteristic polynomial is  $p_5(x)=-x^5+5 x^4-12 x^3+\frac{33 x^2}{2}-\frac{77 x}{6}+\frac{541}{120}$. So, it is clear that the coefficient of $x^1$ is $\frac{F_{5}^{(2)}}{5!}=\frac{77}{6}$. Note that this is the 4-th coefficient of the ordinary generating function of $(\frac{1}{2-e^t})^2$.

\section*{Acknowledgements}

The research of Jos\'e L. Ram\'irez was partially supported by Universidad Nacional de Colombia, Project No. 37805.

\end{document}